\documentclass{amipaper}
\privatetheorem{ques}{Question}{\rm}
\newcommand{\mem}[1]{\in\mathbb{#1}}
\privatetheorem{proposition}{Proposition}{\it}
\privatetheorem{conj}{Conjecture}{\rm}
\privatetheorem{prob}{Problem}{\rm}
\begin{document}
\title{Algebraic and transcendental solutions\\of some exponential equations}
\author{Jonathan Sondow\inst1, Diego Marques\inst2}
\institute{\inst1209 West 97th Street, New York, NY 10025 USA
 \AND \inst2Departamento de Matem\'atica, Universidade de Bras\' ilia, Brazil
           }
\address{\textbf{Jonathan Sondow}\\ 209 West 97th Street, New York, NY 10025 USA\\ e-mail: \texttt{jsondow@alumni.princeton.edu}\AND
          \textbf{Diego Marques}\\ Departamento de Matem\'atica, Universidade de Bras\' ilia, DF, Brazil\\ e-mail: \texttt{diego@mat.unb.br}}
\maketitle

\begin{abstract}
We study algebraic and transcendental powers of
 positive real numbers, including solutions of each of the equations
$x^x=y$, $x^y=y^x$, $x^x=y^y$, $x^y=y$, and $x^{x^y}=y$.
 Applications to values of the iterated
exponential functions are given. The main tools used are classical
 theorems
of
Hermite-Lindemann and Gelfond-Schneider, together with solutions of
exponential Diophantine equations.
\keywords Algebraic, irrational,
transcendental, Gelfond-Schneider Theorem, Hermite-Lindemann
Theorem, iterated exponential.
\AMSclassificationnumber Primary 11J91,
 Secondary
11D61.
\end{abstract}

\section{Introduction}\label{section:1}
Transcendental number theory began in 1844 with Liouville's explicit construction of the first transcendental numbers. In 1872 Hermite proved that $e$ is transcendental, and in 1884 Lindemann extended Hermite's method to prove that $\pi$ is also transcendental. In fact, Lindemann proved a more general result.

\begin{theorem} [Hermite-Lindemann] \label{Lindemann}
The number $e^{\alpha}$ is transcendental for any nonzero algebraic number $\alpha$.
\end{theorem}

As a consequence, the numbers $e^2, e^{\sqrt{2}}$, and $e^{\,i}$ are transcendental, as are $\log 2$ and $\pi$, since $e^{\log 2}=2$ and $e^{\pi i}=-1$ are algebraic.

At the 1900 International Congress of Mathematicians in Paris, as the seventh in his famous list of 23 problems, Hilbert raised the question of the arithmetic nature of the power $\alpha^{\,\beta}$ of two algebraic numbers $\alpha$ and $\beta$. In 1934, Gelfond and Schneider, independently, completely solved the problem (see \cite[p. 9]{baker}).

\begin{theorem}[Gelfond-Schneider] \label{GelSchn}
Assume $\alpha$ and $\beta$ are algebraic numbers, with $\alpha \neq 0$ or $1$, and $\beta$ irrational. Then $\alpha^{\,\beta}$ is transcendental.
\end{theorem}

In particular, $2^{\sqrt{2}}, \sqrt{2}^{\sqrt{2}}$, and $e^{\pi}= i^{\,-2i}$ are all transcendental.

Since transcendental numbers are more ``complicated'' than algebraic irrational ones, we might think that the power of two transcendental numbers is also transcendental, like $e^{\pi}$. However, that is not always the case, as the last two examples for Theorem~\ref{Lindemann} show. In fact, there is no known classification of the power of two transcendental numbers analogous to the Gelfond-Schneider Theorem on the power of two algebraic numbers.

In this paper, we first explore a related question (a sort of converse to one raised by the second author in~\cite[Ap\^{e}ndice B]{marques}).

\begin{ques} \label{qXorYinT}
Given positive real numbers $X\neq1$ and $Y\neq1$, with $X^Y$ algebraic, under which conditions will at least one of the numbers $X,Y$ be transcendental?
\end{ques}

Theorem \ref{GelSchn} gives one such condition, namely, $Y$ irrational. In Sections~2 and~3, we give other conditions for Question~\ref{qXorYinT}, in the case $X^Y=Y^X$. To do this, we use the Gelfond-Schneider Theorem to find algebraic and transcendental solutions to each of the exponential equations $y=x^{\,x}$, $y=x^{\,1/x}$, and $x^{\,y}=y^{\,x}$ with $x\neq y$.

In the Appendix, we study the arithmetic nature of values of three classical infinite power tower functions. We do this by using the Gelfond-Schneider and Hermite-Lindemann Theorems to classify solutions to the equations $y=x^{\,y}$ and $y=x^{\,x^{\,y}}$.

A general reference is Knoebel's Chauvenet Prize-winning article \cite{knoebel}. Consult its very extensive annotated bibliography for additional references and history.\\

\noindent
\textbf{Notation.} We denote by $\mathbb{N}$ the natural numbers, $\mathbb{Z}$ the integers, $\mathbb{Q}$ the rationals, $\mathbb{R}$ the reals,
$\mathbb{A}$ the algebraic numbers, and $\mathbb{T}$ the transcendental numbers. For any set $S$ of complex numbers, $S^+ := S \cap (0,\infty)$ denotes the subset of positive real numbers in $S$. The Fundamental Theorem of Arithmetic is abbreviated FTA.


\section{The case $X=Y$: algebraic numbers $T^T\!$ with $T$ transcendental.}
In this section, we give answers to Question~\ref{qXorYinT} in the case $X=Y$. For this we need a result on the arithmetic nature of $Q^{\,Q}$ when $Q$ is rational.

\begin{lemma} \label{Lemma_Q^Q}
If $Q\mem{Q}\setminus\mathbb{Z}$, then $Q^{\,Q}$ is irrational.
\end{lemma}
\begin{proof} 
If $Q>0$, write $Q = a/b$, where $a,b \mem{N}$ and $\gcd(a,b) = 1$. Set $a_1 = a^a$ and
$b_1 = b^a$. Then $\gcd(a_1,b_1) = 1$ and $(a_1/b_1)^{1/b} = Q^Q \mem{Q}^+$. Using the FTA, we deduce that
$b_1^{1/b} \mem{N}$. We must show that $b = 1$. Suppose on the contrary
 that some prime
$p \ |\ b$. Let $p^n$ be the largest power of $p$
 that
divides $b$. Using $b^{a/b} = b_1^{1/b} \mem{N}$ and the FTA again, we
 deduce
that $p^{na/b} \mem{N}$. Hence $b \ |\ na$. Since $\gcd(a,b)=1$, we get
 $b \
|\ n$. But then $p^n \ |\ n$, contradicting $p^n > n$. Therefore, $b = 1$.

If $Q<0$, write $Q=-a/b$, where $a,b\in \mathbb{N}$ and gcd$(a,b)=1$. If $b$ is odd, then by the previous case, $Q^Q=(-1)^a(a/b)^{-a/b}\notin \mathbb{Q}$. If $b$ is even, then $a$ is odd and $(-1)^{1/b}\notin \mathbb{R}$; hence $Q^Q=(-1)^{1/b}(a/b)^{-a/b}\notin \mathbb{Q}$. This completes the proof.
\end{proof}

As an application, using Theorem \ref{GelSchn} we obtain that $Q^{Q^Q}$ \emph{is transcendental if} $Q\in \mathbb{Q}\setminus\mathbb{Z}$.  

Consider now the equation $x^x=y$. When $0<y<e^{-1/e}=0.69220\dots$, there is no solution $x>0$. If $y=e^{-1/e}$, then $x= e^{-1}=0.36787\dots$. For $y\in(e^{-1/e},1)$, there are exactly two solutions $x_0$ and $x_1$, with $0<x_0<e^{-1}<x_1<1$. (See Figure~\ref{yxx3}, which shows the case $y=1/\sqrt{2}$, $x_0=1/4$, $x_1=1/2$.) Finally, given $y\in [1,\infty)$, there is a unique solution $x\in [1,\infty)$.

\begin{figure}[!htb]
\centering
\includegraphics[scale=0.7]{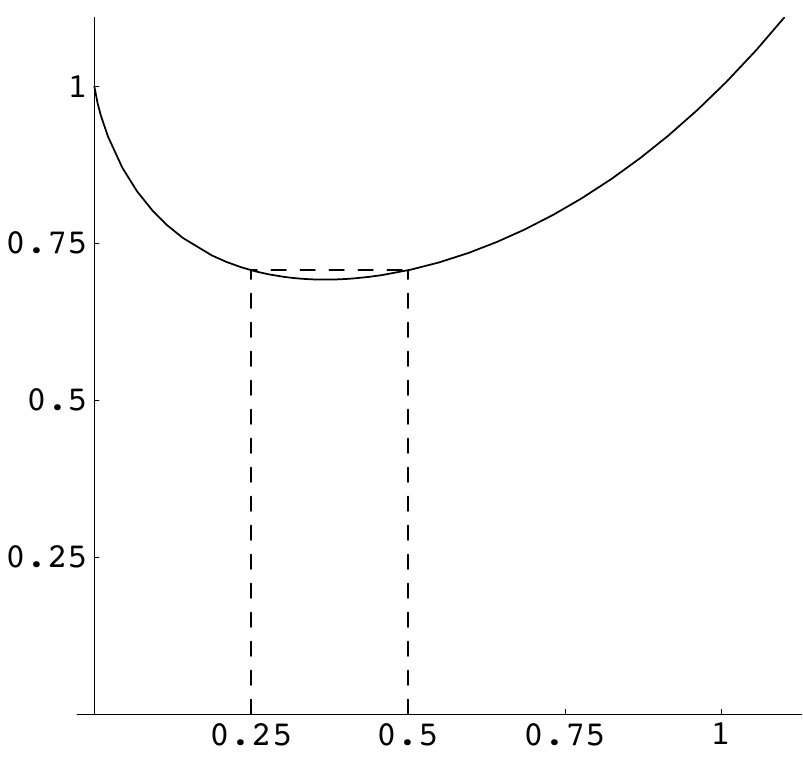}
\caption{$y=x^x$ }
\label{yxx3}
\end{figure}

Turning to the case $X=Y$ of Question~\ref{qXorYinT}, we give two classes of algebraic numbers $A$ such that $T^T=A$ implies $T$ is transcendental.

\begin{proposition} \label{prop1}
Given $A\in [e^{-1/e},\infty)$, let $T\mem{R}^+$ satisfy $T^T=A$. If either

\emph{(i).} $A^n\mem{A}\setminus\mathbb{Q}$ for all $n\mem{N}$, or

\emph{(ii).} $A\mem{Q}\setminus\{n^n:n\mem{N}\}$,\\
then $T$ is transcendental. In particular, $T\mem{T}$ if $T^T\in\mathbb{Q}\cap (e^{-1/e},1)$.
\end{proposition}
\begin{proof}
(i). Suppose $T\mem{A}$. Since $T>0$ and $T^T=A\mem{A}$, Theorem~\ref{GelSchn} implies $T\mem{Q}$, say $T=m/n$ with $m,n\mem{N}$. But then $A^n=T^{\,m}\mem{Q}$, contradicting (i). Therefore, $T\mem{T}$.\\
(ii). Since $T^T=A\mem{Q}\setminus\{n^n:n\mem{N}\}$, Lemma~\ref{Lemma_Q^Q} implies $T$ is irrational. Then Theorem~\ref{GelSchn} yields $T\mem{T}$, and the proposition follows.
\end{proof}

To illustrate case~(i), take $A=\sqrt{3}-1\in(e^{-1/e},1)$. Using a computer algebra system, such as \emph{Mathematica} with its FindRoot command, we solve the equation $x^x=A$ with starting values of $x$ near $0$ and $1$, obtaining the solutions $T_0:=0.18461\dots$ and $T_1:=0.58872\dots$. Similarly, for case~(ii), setting $A=2$ leads to the solution $T_2:=1.55961\dots$. Then
\begin{equation*}
    T_0^{T_0} = T_1^{T_1} = \sqrt{3} - 1, \quad T_0 < e^{-1} < T_1; \qquad T_2^{T_2} = 2; \qquad T_0,T_1,T_2 \mem{T}.
\end{equation*}

\begin{prob} \label{NASC1}
In Proposition~\ref{prop1}, replace the two sufficient conditions (i), (ii) with a necessary and sufficient condition that includes them.
\end{prob}

We will return to the case $X=Y$ of Question~\ref{qXorYinT} at the end of the next section (see Corollary~\ref{F0QQ_trans}).

\section{The case $X^Y=Y^X$, with $X \neq Y$.}
In this section, we give answers to Question~\ref{qXorYinT} by finding algebraic and transcendental solutions of the equation $x^y=y^x$, for positive real numbers $x \neq y$. (Compare Figure~\ref{xyyx2}. Moulton~\cite{moulton} gives a graph for both positive and negative values of $x$ and $y$, and discusses solutions in the complex numbers.)

\begin{figure}[!htb]
\centering
\includegraphics[scale=0.7]{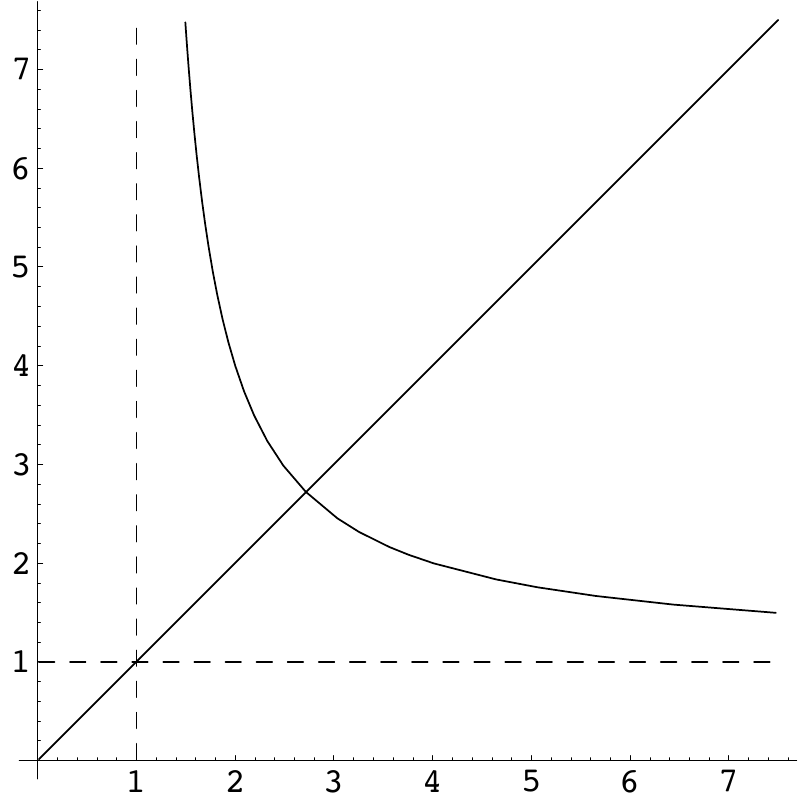}
\caption{$x^y=y^x$ }
\label{xyyx2}
\end{figure}

Consider now Question~\ref{qXorYinT} in the case $X^Y=Y^X=A\mem{A}$, with $X\neq Y$. We give a condition on $A$ which guarantees that at least one of $X,Y$ is transcendental.

\begin{proposition} \label{(i)RT=TR=A}
Assume that
\begin{align} \label{(i)rta}
   T,R&\mem{R}^+, & 
   A&:=T^R=R^T, &
   T &\neq R.
\end{align}
If $A^n\mem{A}\setminus\mathbb{Q}$ for all $n\mem{N}$, then at least one of the numbers $T,R$, say $T$, is transcendental.
\end{proposition}
\begin{proof}
Suppose on the contrary that $T,R \mem{A}$. Since $T^R = R^T =A \mem{A}$ and \eqref{(i)rta} implies $T,R\neq 0$ or $1$, Theorem \ref{GelSchn} yields $T,R \mem{Q}$, say $T=a/b$ and $R=m/n$, where $a, b, m, n \mem{N}$. But then $A^n=(a/b)^m\mem{Q}$, contradicting the hypothesis. Therefore, $\{T,R\}\cap\mathbb{T}\neq\emptyset$.
\end{proof}

In order to give an example of Proposition~\ref{(i)RT=TR=A}, we need the following classical result, which is related to a problem posed in 1728 by D.~Bernoulli~\cite[p.~262]{bernoulli}. (In \cite{knoebel}, see Sections~1 and~3 and the notes to the bibliography.)

\begin{lemma} \label{xy=yx=z}
Given $z \mem{R}^+$, there exist $x$ and $y$ such that
\begin{equation*}
      x^y = y^x = z, \qquad 0<x<y,
\end{equation*}
if and only if $z > e^e = 15.15426 \ldots$. In that case, $1 < x < e < y$ and $x,y$ are given parametrically by
\begin{equation}
      x = x(t):= \left(1+\frac{1}{t}\right)^t, \qquad y = y(t):= \left(1+\frac{1}{t}\right)^{t+1}
\label{XandY}
\end{equation}
for $t>0$. Moreover, $x(t)^{y(t)}$ is decreasing, and any one of the numbers $x\in(1,e)$, $y\in(e,\infty)$, $z\in(e^e, \infty)$, and $t\in(0,\infty)$ determines the other three uniquely.
\end{lemma}
\begin{proof}
Given $x,y \mem{R}^+$ with $x<y$, denote the slope of the line from the origin to the point $(x,y)$ by $s:=y/x$. Then $s>1$, and $y=sx$ gives the equivalences
\begin{align*}
      x^y = y^x &\Longleftrightarrow x^{sx} = (sx)^x \Longleftrightarrow x^s = sx\\
      & \Longleftrightarrow x = x_1(s) := s^{1/(s-1)} \Longleftrightarrow y = y_1(s) := s^{s/(s-1)}.
\end{align*}
The substitution $s=1+t^{-1}$ then produces \eqref{XandY}, implying $1 < x < e < y$. Using L'Hopital's rule, we get
\begin{equation*}
   \lim_{t\to 0^+} x(t)=1, \quad \lim_{t\to 0^+} y(t)=\infty \quad \Longrightarrow \quad \lim_{t\to 0^+} y(t)^{x(t)}=\infty.
\end{equation*}
By calculus, $x(t)$ is increasing, $y(t)$ is decreasing, and $y(t)^{x(t)}\to e^e$ as $t\to\infty$ (see Figure~\ref{x(t)y(t)}). Anderson~\cite[Lemma~4.3]{anderson} proves that the function $y_1(s)^{-x_1(s)}$ is decreasing on the interval $1<s<\infty$, and we infer that $y(t)^{x(t)}$ is decreasing on $0<t<\infty$ (see Figure~\ref{x(t)2}). The lemma follows.
\end{proof}

\begin{figure}[!htb]
\centering
\includegraphics[scale=0.7]{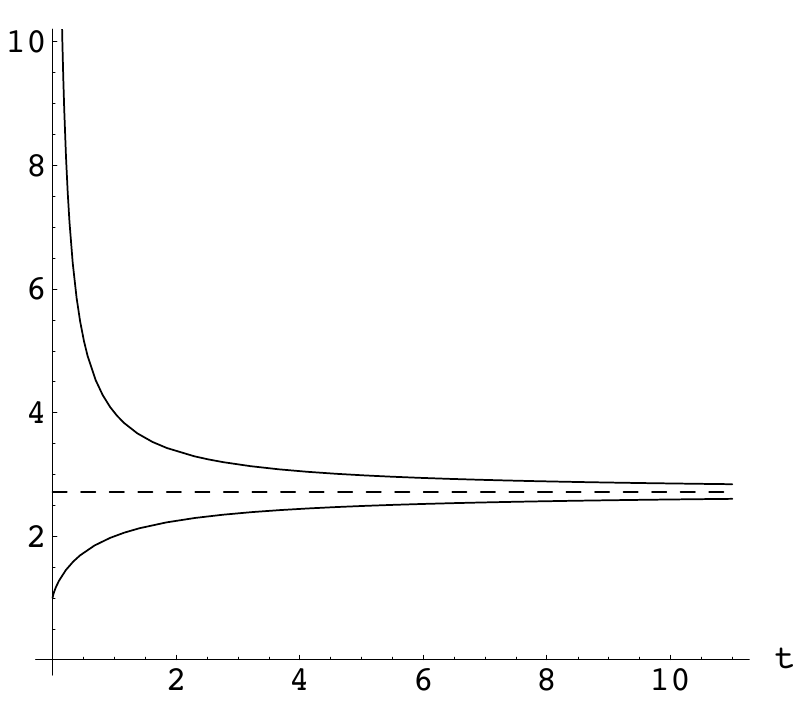}
\caption{The graphs of $x(t)$ (bottom) and $y(t)$}
\label{x(t)y(t)}
\end{figure}

\begin{figure}[!htb]
\centering
\includegraphics[scale=0.7]{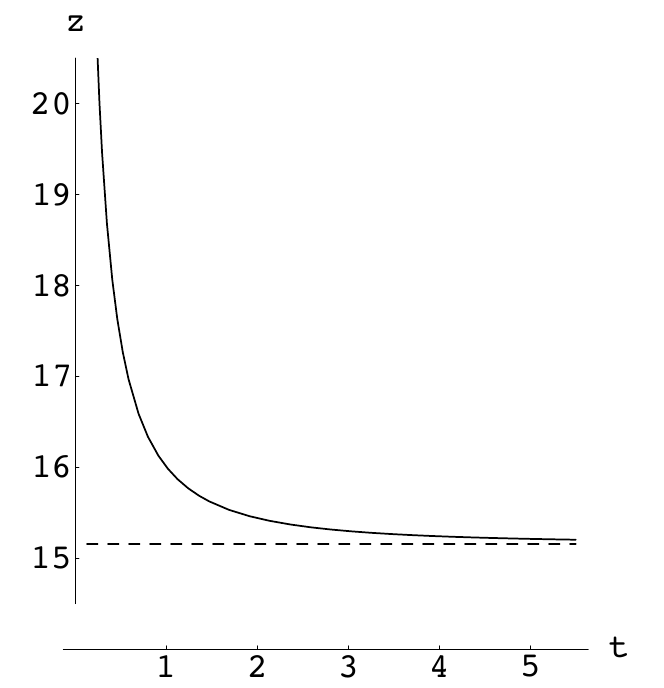}
\caption{$z=x(t)^{y(t)}=y(t)^{x(t)}$}
\label{x(t)2}
\end{figure}

For instance, taking $t=1$ in \eqref{XandY} leads to $2^4=4^2=16$. To parameterize the part of the curve $x^y=y^x$ with $x>y>0$, replace $t$ with $-t-1$ in \eqref{XandY} (or replace $s$ with $1/s$ in the parameterization $x=x_1(s),\ y=y_1(s)$, which is due to Goldbach~\cite[pp.~280-281]{goldbach}). For example, setting $t=-2$ in \eqref{XandY} yields $(x,y)=(4,2)$.

Euler~\cite[pp. 293-295]{euler1} described a different way to find solutions of $x^y = y^x$ with $0<x<y$. Namely, the equivalence
\begin{equation*}
    x^y = y^x \quad \Longleftrightarrow \quad x^{1/x} = y^{1/y}
\end{equation*}
shows that a solution is determined by equal values of the function $g(u) = u^{1/u}$ at $u=x$ and $u=y$. (Figure~\ref{yy14} exhibits the case $x=2$, $y=4$.) From the properties of $g(u)$, including its maximum at $u=e$ and the bound $g(u)>1$ for $u\in(1,\infty)$, we see again that $1<x<e<y$.

\begin{figure}[!htb]
\centering
\includegraphics[scale=0.7]{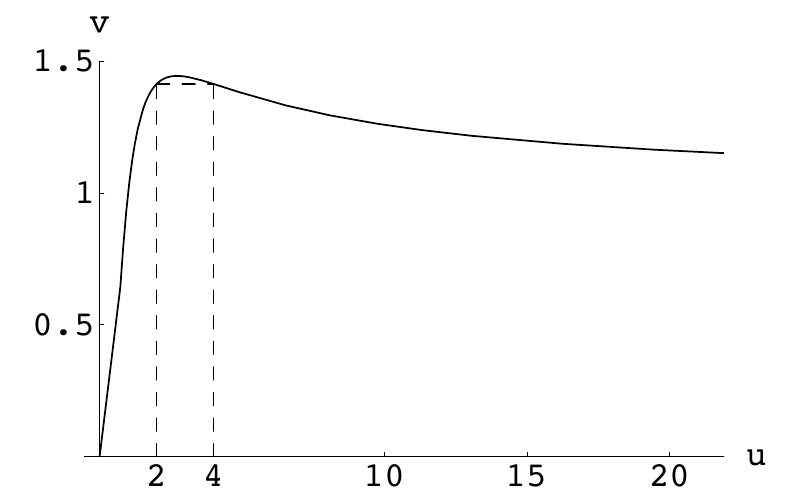}
\caption{$v=g(u)=u^{1/u}$ }
\label{yy14}
\end{figure}

We can now give an example for Proposition~\ref{(i)RT=TR=A}.

\begin{example} \label{exRT=TR=A}
Set $A=14+\sqrt{2}$. Since $A>e^e$, the equation $x(t)^{y(t)}=A$ has a (unique) solution $t=t_1>0$. (Computing $t_1$, we find that $x(t_1)=2.26748\dots$ and $y(t_1)=3.34112\dots$.) Then $(T,R) := (x(t_1),y(t_1))$ or $(y(t_1),x(t_1))$ satisfies
\begin{align*}
    T^R &= R^T = 14+\sqrt{2}, &
    T &\neq R, &
    T &\mem{T}.
\end{align*}
\end{example}

In the next proposition, we characterize the algebraic and rational solutions of $x^y=y^x$ with $0<x<y$. (Part~(i) is due to Mahler and Breusch \cite{mahler}. For other references, as well as all rational solutions to the more general equation $x^y=y^{mx}$, where $m\mem{N}$, see Bennett and Reznick~\cite{bennett}.)

\begin{proposition} \label{onlyAlg}
Assume $0 < A_1 < A_2$. Define $x(t)$ and $y(t)$ as in \eqref{XandY}.

\emph{(i).} Then $A_1^{A_2} = A_2^{A_1}$ and $A_1, A_2 \mem{A}$ if and only if $A_1=x(t)$ and $A_2=y(t)$, with $t\mem{Q}^+$.

\emph{(ii).} In that case, if $t\mem{N}$, then $A_1^{A_2} = A_2^{A_1}\mem{A}$ and $A_1,A_2\mem{Q}$, while if $t\not\mem{N}$, then $A_1^{A_2} = A_2^{A_1}\mem{T}$ and $A_1,A_2\not\mem{Q}$.
\end{proposition}
\begin{proof}
(i). By Lemma~\ref{xy=yx=z}, it suffices to prove that $t\mem{Q}$ if $x(t),y(t)\mem{A}$. Formulas \eqref{XandY} show that $x(t)^{(t+1)/t}=y(t)$. As $x(t)\neq 0$ or $1$, Theorem \ref{GelSchn} implies $t\not\mem{A}\setminus\mathbb{Q}$. From \eqref{XandY} we also see that $y(t)/x(t)=1+t^{-1}$, and hence $t\mem{A}$. Therefore, $t\mem{Q}$.\\
(ii). It suffices to show that if \mbox{$A_1^{A_2} = A_2^{A_1}\mem{A}$}, where $A_1=x(a/b)$ and $A_2=y(a/b)$, with $a,b\mem{N}$ and $\gcd (a,b)=1$, then $b=1$. Theorem \ref{GelSchn} implies $A_1, A_2 \mem{Q}$. It follows, using \eqref{XandY} and the FTA, that $a+b$ and $a$ are $b$th powers, say $a+b=m^b$ and $a=n^b$, where $m,n \mem{N}$. Then $d:=m-n\geq 1$ and
\mbox{$b=(n+d)^b-n^b$} \mbox{$=bn^{b-1}d+\cdots+d^b$}. Hence $b=1$.
\end{proof}

For example, taking $t=2$ and $1/2$ yields
\begin{align*}
    (9/4)^{27/8}&=(27/8)^{9/4}\mem{A}, &
    \sqrt{3}^{\sqrt{27}}&=\sqrt{27}^{\sqrt{3}}\mem{T}.
\end{align*}

Here is another sufficient condition for Question~\ref{qXorYinT} in the case $X^Y=Y^X$ with $X\neq Y$.

\begin{corollary} \label{(ii)RT=TR=A}
Let $T,R\mem{R}^+$ satisfy $T^R=R^T=N\mem{N}$ and $T\neq R$. If $N\neq 16$, then at least one of the numbers $T,R$, say $T$, is transcendental.
\end{corollary}
\begin{proof}
If on the contrary \mbox{$T,R\mem{A}$}, then Proposition~\ref{onlyAlg} implies\linebreak[4] $(T,R)=(x(n),y(n))$ or $(y(n),x(n))$, for some $n\mem{N}$. Thus $x(n)^{y(n)}= N\neq16$. But a glance at Figure~\ref{x(t)2} (or at Lemma~\ref{xy=yx=z}) shows that is impossible.
\end{proof}

For instance, the equation $x(t)^{y(t)}=17$ has a (unique) solution $t=t_1>0$ (computing $t_1$, we get $(x(t_1),y(t_1))=(1.78381\dots,4.89536\dots)$), and for $(T,R) = (x(t_1),y(t_1))$ or $(y(t_1),x(t_1))$ we have
\begin{align*}
    T^R &= R^T = 17, &
    T &\neq R, &
    T &\mem{T}.
\end{align*}

We make the following prediction.

\begin{conj} \label{conj}
In Proposition~\ref{(i)RT=TR=A} and Corollary~\ref{(ii)RT=TR=A} a stronger conclusion holds, namely, that both $T$ and $R$ are transcendental.
\end{conj}

We can give a conditional proof of Conjecture \ref{conj}, assuming a conjecture of Schanuel~\cite[p.~120]{baker}. Namely, in view of Proposition~\ref{(i)RT=TR=A} and Corollary~\ref{(ii)RT=TR=A}, Conjecture \ref{conj} is an immediate consequence of the following conditional result \cite[Theorem 3]{ms}.

\begin{theorem}
Assume Schanuel's conjecture and let $z$ and $w$ be complex numbers, not $0$ or $1$. If $z^{\,w}$ and $w^{\,z}$ are algebraic, then $z$ and $w$ are either both rational or both transcendental.
\end{theorem}

We now give an application of Proposition~\ref{onlyAlg} to Question~\ref{qXorYinT} in the case $X=Y$.

\begin{corollary} \label{F0QQ_trans}
Let $T,Q\in(0,1)$ satisfy $T^{\,T}=Q^{\,Q}$ and $T\neq Q\mem{Q}$. Then $T\mem{T}$ if and only if $x(n) \neq 1/Q \neq y(n)$ for all $n\mem{N}$. In particular, $T\mem{T}$ if $1/Q\mem{N}\setminus\{1,2,4\}$.
\end{corollary}
\begin{proof}
It is easy to see the equivalences
\begin{equation*}
   T^{\,T}=Q^{\,Q} \quad \Longleftrightarrow \quad (1/T)^{1/Q}=(1/Q)^{1/T}
\end{equation*}
and, as $\mathbb{A}$ is a field, $T\mem{T} \Longleftrightarrow 
1/T\mem{T}$. Using Proposition~\ref{onlyAlg}, the ``if and only if'' statement follows. Since $n\mem{N}$ and $1/Q\mem{N}\setminus\{2,4\}$ imply $x(n)\neq 1/Q\neq y(n)$, the final statement also holds. 
\end{proof}

For example, taking $Q=4/9=1/x(2)$ leads to $(4/9)^{4/9}=(8/27)^{8/27}\mem{A}$, while $Q=1/3$ and $2/3$ give
\begin{align*}
    (1/3)^{1/3} &= T_1^{\,T_1}, \qquad T_1 \mem{T};&
    (2/3)^{2/3} &= T_2^{\,T_2}, \qquad T_2 \mem{T}.
\end{align*}
Here $T_1=0.40354\dots$ and $T_2=0.13497\dots$ can be calculated by computing solutions to the equations $x^x=(1/3)^{1/3}$ and $x^x=(2/3)^{2/3}$, using starting values of $x$ in the intervals $(e^{-1},1)$ and $(0,e^{-1})$, respectively.

\section{Appendix: The infinite power tower functions}
We use the Gelfond-Schneider and Hermite-Lindemann Theorems to find algebraic, irrational, and transcendental values of three classical functions, whose analytic properties were studied by Euler~\cite{euler2}, Eisenstein~\cite{eisenstein}, and many others.

\begin{definition} \label{defh}
The \emph{infinite power tower} (or \emph{iterated exponential}) \emph{function} $h(x)$ is the limit of the sequence of \emph{finite power towers} (or \emph{hyperpowers}) $x,\ x^{\,x},\ x^{\,x^{\,x}},\ \ldots$. For $x>0$, the sequence converges if and only if (see \cite{anderson}, Cho and Park~\cite{cho}, De Villiers and Robinson~\cite{devilliers}, Finch~\cite[p. 448]{finch}, and~\cite{knoebel})
\begin{equation*} 
    0.06598\ldots = e^{-e} \leq x \leq e^{1/e} = 1.44466\ldots,
\end{equation*}
and in that case we write
\begin{equation*}
   h(x) = x^{\,x^{\,x^{\,\cdot^{\cdot^{\cdot}}}}}.
\end{equation*}
\end{definition}

By substitution, we see that $h$ satisfies the identity

\begin{equation}
   x^{\,h(x)} = h(x). \label{eqh=x^h}
\end{equation}

\noindent
Thus $y=h(x)$ is a solution of the equations $x^y=y$ and, hence, $x=y^{\,1/y}$. In other words, $g(h(x)) = x$, where $g(u)=u^{\,1/u}$ for $u>0$. Replacing $x$ with $g(x)$, we get $g(h(g(x)))=g(x)$ if $g(x)\in[e^{-e},e^{1/e}]$. Since $g$ is one-to-one on $(0,e]$, and since $h$ is bounded above by $e$ (see~\cite{knoebel} for a proof) and $g([e,\infty))\subset(1,e^{1/e}]$ (see Figure~\ref{yy14}), it follows that
\begin{align}
   h(g(x)) &= x \qquad (e^{-1}\leq x\leq e), &
   h(g(x)) < x \qquad (e<x<\infty). \label{h=,h<}
\end{align}
Therefore, $h$ is a partial inverse of $g$, and is a bijection (see Figure~\ref{hx})
\begin{equation*}
    h:[e^{-e},e^{1/e}]\ \to[e^{-1},e].
\end{equation*}

\begin{figure}[!htb]
\centering
\includegraphics[scale=0.7]{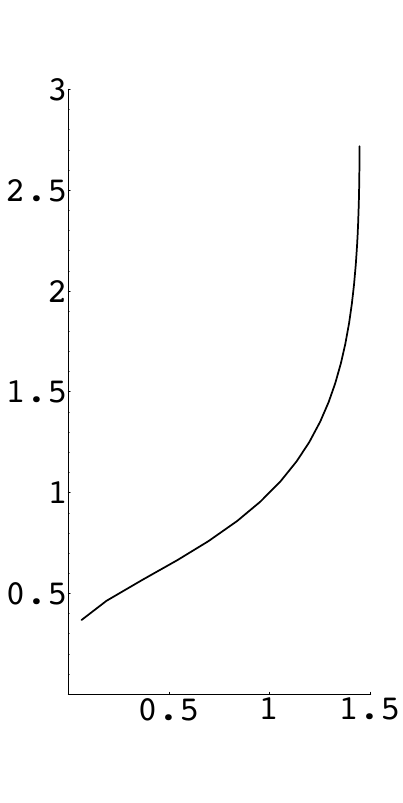}
\caption{$y=h(x) = x^{\,x^{\,x^{\,\cdot^{\cdot^{\cdot}}}}}$ }
\label{hx}
\end{figure}

For example, taking $x=1/2$ and $2$ in \eqref{h=,h<} gives
\begin{align} \label{Ex_h(A)inA}
(1/4)^{(1/4)^{(1/4)^{\cdot^{\cdot^{\cdot}}}}}&=\frac12, & \sqrt{2}^{\sqrt{2}^{\sqrt{2}^{\cdot^{\cdot^{\cdot}}}}}&=2,
\end{align}
while choosing $x=3$ yields
\begin{equation*}
   \sqrt[3]{3}^{\sqrt[3]{3}^{\sqrt[3]{3}^{\cdot^{\cdot^{\cdot}}}}} < 3.
\end{equation*} 

Recall that the Hermite-Lindemann Theorem says that if $A$ is any nonzero algebraic number, then $e^A$ is transcendental. We claim that \emph{if in addition $A$ lies in the interval $(-e,e^{-1})$, then $h(e^A)$ is also transcendental}. To see this, set $x=e^A$ and $y=h(x)$. Then \eqref{eqh=x^h} yields $e^{Ay}=y$, and Theorem \ref{Lindemann} implies $y\mem{T}$, proving the claim. For instance,
\begin{equation} \label{Ex_h(T)inT}
   \sqrt[3]{e}^{\sqrt[3]{e}^{\sqrt[3]{e}^{\cdot^{\cdot^{\cdot}}}}} = 1.85718\dots\mem{T},
\end{equation} 
where the value of $h(\sqrt[3]{e})$ can be obtained by computing a solution to $x^{1/x}=\sqrt[3]{e}$, using a starting value of $x$ between $e^{-1}$ and $e$.

Here is an application of Proposition~\ref{prop1}.

\begin{corollary} \label{Prop_h(A)inT}
Given $A\in\left[e^{-e},e^{1/e}\right]$, if either $A^n\mem{A}\setminus\mathbb{Q}$ for all $n\mem{N}$, or $A\mem{Q}\setminus \{1/4,1\}$, then
\begin{equation}\label{Ex_h(A)inT}
   A^{A^{A^{\cdot^{\cdot^{\cdot}}}}}\mem{T}.
\end{equation} 
\end{corollary}
\begin{proof}
From \eqref{eqh=x^h}, we have $A_1 := 1/A=\left(1/h(A)\right)^{1/h(A)}$. The hypotheses imply that $A_1$ satisfies condition (i) or (ii) of Proposition~\ref{prop1}. Thus $1/h(A)$ and, hence, $h(A)$ are transcendental.
\end{proof}

For example, $h\left((\sqrt{2}+1)/2\right)=1.27005\dots\mem{T}$ and
\begin{equation*}
   (1/2)^{(1/2)^{(1/2)^{\cdot^{\cdot^{\cdot}}}}}=0.64118\dots\mem{T}.
\end{equation*} 

It is easy to give an infinite power tower analog to the examples in Section 2 of powers $T^T\mem{A}$ with $T\mem{T}$. Indeed, Theorem~\ref{GelSchn} and relation~\eqref{eqh=x^h} imply that \emph{if \mbox{$A \in (\mathbb{A}\setminus\mathbb{Q})\cap (e^{-1},e)$}, then}
\begin{align} \label{Ex_h(T)inA}
   T&:=1/A^{A}\mem{T}, &
   T^{T^{T^{\cdot^{\cdot^{\cdot}}}}}&= 1/A\mem{A}.
\end{align} 

Notice that \eqref{Ex_h(A)inA}, \eqref{Ex_h(T)inT}, \eqref{Ex_h(A)inT}, \eqref{Ex_h(T)inA} represent the four possible cases $(x,h(x))\in\mathbb{A}\times\mathbb{A}$, $\mathbb{T}\times\mathbb{T}$, $\mathbb{A}\times\mathbb{T}$, $\mathbb{T}\times\mathbb{A}$, respectively.

We now define two functions each of which extends $h$ to a larger domain.

\begin{definition} \label{defhohe}
The \emph{odd infinite power
tower function} $h_o(x)$ is the limit of the sequence of finite power towers of odd height:
\begin{equation*}
  x,\ x^{\,x^{\,x}},\ x^{\,x^{\,x^{\,x^{\,x}}}},\ \ldots\ \longrightarrow\ h_o(x).
\end{equation*}
Similarly, the \emph{even infinite power
tower function} $h_e(x)$ is defined as the limit of the sequence of finite power towers of even height:
\begin{equation*}
  x^{\,x},\ x^{\,x^{\,x^{\,x}}},\ x^{\,x^{\,x^{\,x^{\,x^{\,x}}}}},\ \ldots\ \longrightarrow\ h_e(x).
\end{equation*}
Both sequences converge on the interval $0 < x \leq e^{\,1/e}$ (for a proof, see~\cite{anderson} or~\cite{knoebel}).
\end{definition}

It follows from Definition~\ref{defhohe} that $h_o$ and $h_e$ satisfy the identities
\begin{align} \label{ho=x^x^ho}
   x^{\,x^{\,h_o(x)}}&= h_o(x), &
   x^{\,x^{\,h_e(x)}}&= h_e(x)
\end{align}
and the relations
\begin{align} \label{ho=x^he}
   x^{\,h_e(x)}&= h_o(x), &
   x^{\,h_o(x)}&= h_e(x)
\end{align}
on $(0,e^{\,1/e}]$. From \eqref{ho=x^x^ho}, we see that $y=h_o(x)$ and $y=h_e(x)$ are solutions of the equation $y=x^{\,x^{\,y}}$. So is $y=h(x)$, since $y=x^y$ implies $y=x^{\,x^{\,y}}$.

It is proved in \cite{anderson} and \cite{knoebel} that on the subinterval $[e^{-e},e^{1/e}]\subset(0,e^{\,1/e}]$ the three infinite power tower functions $h,\ h_o,\ h_e$ are all defined and are equal, but on the subinterval $\left(0, e^{-e}\right)$ only $h_o$ and $h_e$ are defined, and they satisfy the inequality
\begin{equation}
   h_o(x)<h_e(x) \qquad (0<x<e^{-e})
\label{ho<he}
\end{equation}
and are surjections (see Figure~\ref{knoebel})
\begin{align*}
   h_o:\ &(0,e^{1/e}] \to (0,e], &
   h_e:\ &(0,e^{1/e}] \to [e^{-1},e].
\end{align*}

\begin{figure}[!htb]
\centering
\includegraphics[scale=0.6]{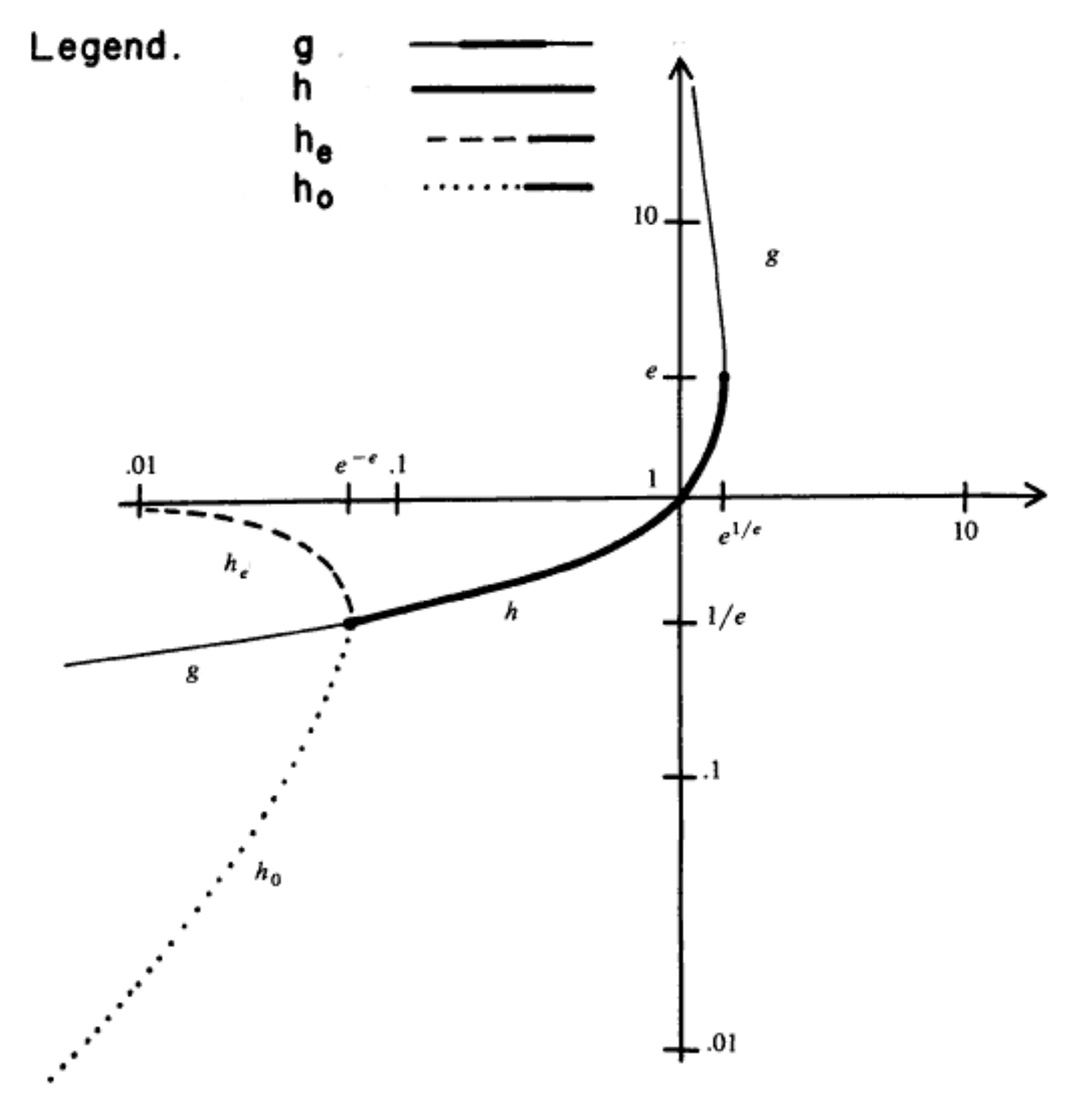}
\caption{(from \cite{knoebel}) $x=g(y)$, $y=h(x)$, $y=h_e(x)$, $y=h_o(x)$ }
\label{knoebel}
\end{figure}

In order to give an analog for $h_o$ and $h_e$ to Corollary~\ref{Prop_h(A)inT} on $h$, we require a lemma.

\begin{lemma} \label{QQQ'}
Assume $Q,Q_1\mem{Q}^+$. Then
\begin{equation}
   Q^{\,Q^{\,Q_1}}=Q_1 \label{eqQQQ'}
\end{equation}
if and only if $(Q,Q_1)$ is equal to either $(1/16,1/2)$ or $(1/16,1/4)$ or $(1/n^{\,n},1/n)$, for some $n\mem{N}$.
\end{lemma}
\begin{proof}
The ``if'' part is easily verified. To prove the ``only if'' part, note first that \eqref{eqQQQ'} and Theorem~\ref{GelSchn} imply $Q^{\,Q_1}\mem{Q}$. Then, writing $Q=a/b$ and $Q_1=m/n$, where $a,b,m,n\mem{N}$ and $\gcd(a,b) = \gcd(m,n) = 1$, the FTA implies $a = a_1^{\,n}$ and $b = b_1^{\,n}$, for some $a_1,b_1\mem{N}$. From \eqref{eqQQQ'} we infer that $m^{\,b_1^{\,m}} = a_1^{\,na_1^{\,m}}$ and $n^{\,b_1^{\,m}} = b_1^{\,na_1^{\,m}}$.\\
\indent
We show that $m = 1$. If $m \neq 1$, then some prime $p \mid m$, and hence $p \mid a_1$. Write $m = m'p^{\,r}$ and $a_1 = a_2p^{\,s}$, where $r,s\mem{N}$ and $\gcd(m',p) = \gcd(a_2,p) = 1$. Substituting into $m^{\,b_1^{\,m}} = a_1^{\,na_1^{\,m}}$, we deduce that $rb_1^{\,m} = sna_1^{\,m}$. Since $\gcd(a_1,b_1)=1$, we have $a_1^{\,m} \mid r$. But $a_1^{\,m}= a_1^{\,m'p^{\,r}}>r$, a contradiction. Therefore, $m = 1$.\\
\indent
It follows that $a_1 = 1$, and hence $n^{\,b_1} = b_1^{\,n}$.  Proposition~\ref{onlyAlg} then implies that $(n,b_1) = (2,4)$ or $(n,b_1) = (4,2)$ or $n = b_1$. The lemma follows.
\end{proof}

\begin{proposition} \label{hohe}
We have $h_o(1/16)=1/4$ and $h_e(1/16)=1/2$. On the other hand, if $Q\mem{Q}\cap\left(0,e^{-e}\right]$ but $Q\neq 1/16$, then $h_o(Q)$ and $h_e(Q)$ are both irrational, and at least one of them is transcendental.
\end{proposition}
\begin{proof}
Since $1/16<e^{-e}$, the equation
\begin{equation*}
   (1/16)^{\,(1/16)^{\,y}}=y
\end{equation*}
has exactly three solutions (see~\cite{knoebel} and Figure~\ref{knoebel}), namely, $y= 1/4,\ 1/2$, and $y_0$, say, where $1/4<y_0<1/2$. By~\eqref{ho=x^x^ho} and \eqref{ho<he}, two of the solutions are $y=h_o(1/16)$ and $h_e(1/16)$. In view of~\eqref{ho<he}, either $h_o(1/16)=1/4$ or $h_o(1/16)=y_0$. But the latter would imply that $h_e(1/16)=1/2$, which leads by \eqref{ho=x^he} to $y_0=(1/16)^{\,1/2}=1/4$, a contradiction. Therefore $h_o(1/16)=1/4$. Then \eqref{ho=x^he} implies $h_e(1/16)=(1/16)^{\,1/4}=1/2$, proving the first statement.\\
\indent
To prove the second, suppose $Q_1:=h_o(Q)$ is rational. Then \eqref{ho=x^x^ho} and\linebreak[4] Lemma~\ref{QQQ'} imply $(Q,Q_1)=(1/n^{\,n},1/n)$, for some $n\mem{N}$. Hence $Q^{\,Q_1}=Q_1$. But from \eqref{ho=x^he} and \eqref{ho<he} we see that $Q^{\, h_o(Q)}= h_e(Q)>h_o(Q)$, so that $Q^{\,Q_1}>Q_1$, a contradiction. Therefore, $h_o(Q)$ is irrational. The proof that $h_e(Q)\not\mem{Q}$ is similar. Now \eqref{ho=x^he} and Theorem~\ref{GelSchn} imply that $\{h_o(Q),h_e(Q)\}\cap\mathbb{T}\neq\emptyset$.
\end{proof}

For example, the numbers $h_o(1/17)=0.20427\dots$ and $h_e(1/17)=0.56059\dots$ are both irrational, and at least one is transcendental. The values were computed directly from Definition~\ref{defhohe}.

\begin{conj} \label{hoheconj}
In the second part of Proposition \ref{hohe} a stronger conclusion holds, namely, that both $h_o(Q)$ and $h_e(Q)$ are transcendental.
\end{conj}

As with Conjecture \ref{conj}, we can give a conditional proof of Conjecture \ref{hoheconj}. Namely, in view of Proposition~\ref{hohe} and the identities~\eqref{ho=x^x^ho}, Conjecture \ref{hoheconj} is a special case of the following conditional result~\cite[Theorem~4]{ms}.

\begin{theorem}
Assume Schanuel's conjecture and let $\alpha\neq0$~and~$z$ be complex numbers, with $\alpha$ algebraic and $z$~irrational. If $\alpha^{\,\alpha^{\,z}}=z,$ then $z$ is transcendental.
\end{theorem}

Some of our results on the arithmetic nature of values of $h,\ h_o$, and $h_e$ can be extended to other positive solutions to the equations $y=x^{\,y}$ and $y=x^{\,x^{\,y}}$. As with the rest of the paper, an extension to negative and complex solutions is an open problem (compare~\cite[Section~4]{knoebel} and~\cite{moulton}).

\paragraph{Acknowledgments.} We are grateful to Florian Luca, Wadim Zudilin, and the anonymous referee for valuable comments.

\end{document}